\documentclass[12pt]{article}

\usepackage{amsmath,amsthm,amssymb}
\usepackage{color,graphicx}
\usepackage{tikz}
\usepackage[noadjust,sort]{cite}

\usepackage[colorlinks=true,citecolor=black,linkcolor=black,urlcolor=blue]{hyperref}
\usepackage{mathtools}

\normalsize

\numberwithin{equation}{section}

\def\nfrac#1#2{{\textstyle\frac{#1}{#2}}}
\def\({\bigl(}
\def\){\bigr)}

\newtheorem{thm}{Theorem}[section]
\newtheorem{cor}[thm]{Corollary}
\newtheorem{lemma}[thm]{Lemma}

\theoremstyle{definition}
\newtheorem*{remark*}{Remark}

\usepackage{float, caption}
\newfloat{algorithm}{ht!}{lop}
\floatname{algorithm}{Algorithm}   

\newcommand{\widesim}[2][1.3]{
  \mathrel{\overset{#2}{\scalebox{#1}[1]{$\sim$}}}
}
\def\simj#1{\widesim{#1}}
\def\abs#1{\lvert#1\rvert} \let\card=\abs
 
\def\dfrac#1#2{\lower0.15ex\hbox{\large$\textstyle\frac{#1}{#2}$}}
\def\({\bigl(}
\def\){\bigr)}
\def\st{\mathrel{|}}

\let\eps=\varepsilon

\def\X{\boldsymbol{X}}

\def\dvec{{\boldsymbol{d}}}
\def\xvec{{\boldsymbol{x}}}

\def\bvec{{\boldsymbol{b}}}

\def\yvec{{\boldsymbol{y}}}

\def\gvec{{\boldsymbol{g}}}
\def\a{{\phi_{\mathrm{min}}}}  
\def\b{{\phi_{\mx}}} 
\def\Hd{{H_{\dvec}}} 

\def\calF{\mathcal{F}}
\def\calG{\mathcal{G}}

\def\calT{\mathcal{T}}

\def\mx{{\mathrm{max}}}
\def\dmax{d_{\mx}}
\def\gmax{g_{\mx}}

\def\xmax{x_{\mx}}

\def\eg{\beta}  
\def\barg{\bar{\mu}}

\def\E{\operatorname{\mathbb{E}}}

\def\Reals{{\mathbb{R}}}

\allowdisplaybreaks

\title{The average number of spanning trees in\\ sparse graphs with given degrees\thanks{Research supported by the Australian Research Council, Discovery Project DP140101519.}} 

\author{
Catherine Greenhill\\
\small School of Mathematics and Statistics\\[-0.8ex]
\small UNSW Australia\\[-0.8ex]
\small Sydney NSW 2052, Australia\\
\small \tt c.greenhill@unsw.edu.au\\
\and
Mikhail Isaev\\
\small Research School of Computer Science\\[-0.8ex]
\small Australian National University\\[-0.8ex]
\small Canberra, ACT 2601, Australia\\
\small Moscow Institute of Physics and Technology\\[-0.8ex]
\small Dolgoprudny, 141700, Russia\\
\small\tt isaev.m.i@gmail.com 
\and
Matthew Kwan\\
\small Department of Mathematics\\[-0.8ex]
\small ETH Z{\" u}rich\\[-0.8ex]
\small Z{\" u}rich, 8092, Switzerland\\[-0.8ex]
\small\tt matthew.kwan@math.ethz.ch\\
\and
Brendan D. McKay\\
\small Research School of Computer Science\\[-0.8ex]
\small Australian National University\\[-0.8ex]
\small Canberra, ACT 2601, Australia\\
\small\tt brendan.mckay@anu.edu.au
}

\date{20 February 2017}

\begin{document}

\maketitle

\begin{abstract}
We give an asymptotic expression for the expected number of spanning trees
in a random graph with a given degree sequence $\dvec=(d_1,\ldots, d_n)$,
provided that the number of edges is at least $n + \nfrac{1}{2} \dmax^4$,
where $\dmax$ is the maximum degree.
A key part of our argument involves establishing a concentration result for
a certain family of functions over random trees with given degrees, using Pr{\" u}fer
codes. 
\end{abstract}

\section{Introduction}\label{s:intro}

The number of spanning trees $\tau(G)$ in a graph $G$ (also called the \emph{complexity} of $G$) is an important graph parameter that has connections to a wide range of topics, including the study of electrical networks, algebraic graph theory, statistical physics and number theory (see for example \cite{Alo90,MW99,Sha87,Wu77}). These connections are largely related to the \emph{matrix tree theorem}, which says that $\tau(G)$ is equal to any cofactor of the Laplacian matrix of $G$.

There is a large body of existing work concerning the approximate value of 
$\tau (G)$ for graphs with given degree sequences, and random graphs with
given degree sequences, especially in the regular case. 
Let $\dvec=(d_1,\ldots,d_n)$ be a vector of positive integers with even sum, and
let $\Gamma_{\dvec}$ denote the set of all graphs on the vertex set $\{1,2,\ldots, n\}$ with degree sequence $\dvec$. 
If~every entry of $\dvec$ equals $d$ then we
write $\Gamma_{n,d}$ for the set of all $d$-regular graphs on $\{ 1,2, \ldots, n\}$.
Let $\calG_{\dvec}$ be the random graph with degree sequence $\dvec$, chosen
uniformly at random from $\Gamma_{\dvec}$, and let $\calG_{n,d}$ be the random
$d$-regular graph on vertex set $\{ 1,2,\ldots, n\}$, chosen uniformly at random  from $\Gamma_{n,d}$.
Unless otherwise stated, all asymptotics in this paper hold as $n\to\infty$,
possibly along some infinite subsequence of $\mathbb{N}$.

The number of spanning trees in a graph is strongly controlled by its degree sequence. 
Let
\[ d = \frac 1n\sum_j d_j, \qquad
  \hat d = \biggl(\,\prod_{j=1}^n d_j\biggr)^{\!1/n}\]
denote the arithmetic and geometric means of the degree sequence $\dvec$.
The best uniform upper bound for regular graphs is due to McKay \cite{McK83}, who proved that when $d\geq 3$,
\[
\tau(G)=O(1) \biggl( \frac{(d-1)^{d-1}}{(d^2-2d)^{d/2-1}}\biggr)^{\!n}
                \frac{\log n}{nd\log d}
\]
for all $G\in \Gamma_{n,d}$.
This was proved sharp within a constant by Chung and Yau~\cite{chung}.
Kostochka~\cite{Kos95} proved that 
\[ \(\hat d(1-\varepsilon)\)^n \leq \tau(G)\leq \frac{\hat{d}^n}{n-1}\]
 for any connected $G\in \Gamma_{\dvec}$, 
where $\varepsilon =\varepsilon(\delta)>0$ tends to zero 
as $\delta=\min_j d_j\to \infty$.  This lower bound extended a result of 
Alon~\cite{Alo90} on $\tau(G)$ in the case of $d$-regular graphs.

To discuss random graphs, define the random variables 
\[ \tau_\dvec=\tau(\calG_{\dvec}) \,\, \text{ and } \,\,
   \tau_{n,d} = \tau(\calG_{n,d}).\]
That is, $\tau_{\dvec}$ is the 
number of spanning trees in $\calG_{\dvec}$, and $\tau_{n,d}$ is the number of
spanning trees in $\calG_{n,d}$. 
McKay \cite{McK81} proved that for fixed $d$, 
\[
\tau_{n,d}^{1/n}\to \frac{(d-1)^{d-1}}{(d^2-2d)^{d/2-1}}
\]
with probability 1. 
An alternative proof in a much more general framework was given by Lyons in \cite[Example~3.16]{Lyo05}.

McKay~\cite{subgraphs} gave the expected value $\E \tau_\dvec$ to within a constant 
factor, in the case that $d_j = O(1)$ for all~$j$ and the average degree is at least 
$2+\eps$, for some $\eps > 0$.
Specifically, McKay proved that under these conditions,
the expected number of spanning trees is
\begin{equation}
\label{McKay}
    \E \tau_\dvec=\Theta(1) \;\frac 1n \biggl( \frac{\hat d\,(d-1)^{d-1}}
                                   {d^{d/2} (d-2)^{d/2-1}}
                           \biggr)^{\!n}.
\end{equation}
Greenhill, Kwan and Wind~\cite{GKW} recently
found the asymptotic value of this $\Theta(1)$ factor, for 
random $d$-regular graphs with $3\le d = O(1)$. Specifically,
they proved that the  $\Theta(1)$ in (\ref{McKay}) is asymptotic to the constant 
\begin{equation}\label{GKW-old}
    \frac{(d-1)^{1/2}}{(d-2)^{3/2}}\,
   \exp\biggl( \frac{6d^2-14d+7}{4(d-1)^2} \biggr).
\end{equation}
(This is about $e^{3/2}/d$ for large $d$.)
They also gave the asymptotic distribution of the number
of spanning trees in a random cubic graph.

In this paper, we obtain an asymptotic expression for $\E \tau_\dvec$ for a 
wider range of sparse degree sequences $\dvec$ than in any of the above 
random graph results.

Our main result is the following. 

\begin{thm}
\label{main-sparse}
Let $\dvec = \dvec(n)=(d_1,\ldots, d_n)$ be a vector of positive integers
with even sum, for every $n$ in some infinite subsequence of $\mathbb{N}$. 
Define
\begin{align*}
d_\mx &= \max_j d_j,\qquad d = \dfrac{1}{n}\, \sum_{j=1}^n d_j, \qquad
\hat d = \biggl(\,\prod_{j=1}^n d_j\biggr)^{\!1/n}, \qquad
R =  \dfrac{1}{n}\, \sum_{j=1}^n (d_j - d)^2
\end{align*}
and let
\[ \Hd =
 \frac{(d-1)^{1/2}}{(d-2)^{3/2}\, n}\,
     \biggl( \frac{\hat{d}\, (d-1)^{d-1}}{d^{d/2}\, (d-2)^{d/2-1}}\biggr)^{\!n} .
\]
Suppose that $d_\mx^4 \leq (d-2) n$.
Then the sequence $\dvec$ is graphical for sufficiently large $n$, and
the expected number of spanning trees in $\mathcal{G}_{\dvec}$ is given by
\[
\E\tau_{\dvec}
 =  \Hd\,
\exp\biggl( \frac{6d^2-14d+7}{4(d-1)^2}
       + \frac{R}{2(d-1)^3}
       + \frac{(2d^2-4d+1)R^2}{4(d-1)^4\, d^2} 
    + O\biggl( \frac{d_\mx^4}{(d-2)n} + \eta\biggr)\biggr),
\]
where
\begin{align*}
  \eta &= \min\biggl\{ \frac{d_\mx^4}{(d-2)^2n},\,
                                 \frac{d_\mx^3\log n}{(d-2)n},\,
                                 d_\mx(d-2)\biggr\} \\
         &= O\biggl( \frac{d_\mx^4}{(d-2)n} + \frac{(\log n)^{5/2}}{n^{1/2}}\biggr).
\end{align*}
\end{thm}

Some remarks about this result are given below.
\begin{itemize}
\item
Due to the Erd{\H o}s-Gallai Theorem,
under the conditions of Theorem~\ref{main-sparse} the sequence $\dvec$ is always graphical 
(without any requirement for $n$ to be large). 
Since this fact is not required for our asymptotic formula, we omit the proof.
\item
Since $\dmax \geq 1$, the condition $\dmax^4\leq (d-2) n$ 
implies that $d > 2$.
\item
Other than the relative error term, the expression given by Theorem~\ref{main-sparse} matches 
\eqref{GKW-old} in the regular case,
showing that the formula obtained in~\eqref{GKW-old} for regular graphs with 
constant~$d\geq 3$
also holds for $d$-regular graphs with slowly growing $d$ (in particular, it holds when $d=o(n^{1/3})$).
\item
Under our assumptions, the relative error term may not be vanishing,
though it is always bounded.  Let $m = \nfrac{1}{2}\sum_{j=1}^n d_j$ be the number of
edges in any graph in $\Gamma_{\dvec}$. The condition $\dmax^4 \leq (d-2)n$ is 
equivalent to the condition that $m\geq n + \nfrac{1}{2} \dmax^4$, or in other words,
that there are at least $\nfrac{1}{2}\dmax^4 + 1$ more edges in any graph in 
$\Gamma_{\dvec}$ than in a tree on $n$ vertices.  For example, when $\dmax = 3$, 
our result holds with a bounded error if the number of edges exceeds $n-1$ by at least
42.
\end{itemize}

In particular, we have the following corollary when $d$ is close to 2.

\begin{cor}
Suppose that $d=2+2x/n$ where
$\nfrac{1}{2}\dmax^4 \leq x\leq n^{1/2}$.
(This corresponds to graphs with $n+x$ edges.)
Then
\begin{align*}
\E \tau_{\dvec} &= \frac1n\left(\frac{e}{2}\right)^x\, \left(\frac{n}{2x}\right)^{3/2 + x}
   \, \biggl(\frac{\hat{d}}{2}\,\biggr)^n\,
  \exp\left(\frac{(6+R)(2+R)}{16} 
+ \frac{3x^2}{2n} +
            O\left(\frac{\dmax^4}{x} + \frac{x^3}{n^2}\right)\right).
\end{align*}
\label{close-to-2}
\end{cor}

\begin{proof}
We estimate the various terms in Theorem~\ref{main-sparse}. 
First note that 
\[ \left(d-1\right)^{1/2}=\left(1+2x/n\right)^{1/2}=e^{O\left(x/n\right)}\]
and 
\[
\frac{1}{(d-2)^{3/2}}\biggl(\frac{\hat{d}}{(d-2)^{d/2-1}}\biggr)^{n}=\hat{d}^{n}\left(\frac{n}{2x}\right)^{3/2+x}.
\]
Next, a series expansion yields
\begin{align*}
\log\biggl(\frac{(d-1)^{d-1}}{d^{d/2}}\biggr) & =\left(1+2x/n\right)\log\left(1+\frac{2x}{n}\right)-\left(1+x/n\right)\log\left(2+\frac{2x}{n}\right)\\
 & =-\log2+\left(1-\log2\right)\frac{x}{n}+\frac{3}{2}\left(\frac{x}{n}\right)^{2}+O\left(\left(\frac{x}{n}\right)^{3}\right)
\end{align*}
so we have
\[
\biggl(\frac{(d-1)^{d-1}}{d^{d/2}}\biggr)^{n}=2^{-n}\left(\frac{e}{2}\right)^{x}\exp\left(\frac{3x^{2}}{2n}+O\left(\frac{x^{3}}{n^{2}}\right)\right).
\]
Then, we can compute
\begin{align*}
\frac{6d^{2}-14d+7}{4\left(d-1\right)^{2}} & =\frac{3}{4}+O\left(\frac{x}{n}\right),\\
\frac{1}{2\left(d-1\right)^{3}} & =\frac{1}{2}+O\left(\frac{x}{n}\right),\\
\frac{\left(2d^{2}-4d+1\right)}{4\left(d-1\right)^{4}d^{2}} &= \frac{1}{16}+O\left(\frac{x}{n}\right).
\end{align*}
So, noting that $R\le\dmax^{2}$, we have 
\begin{align}
\frac{6d^{2}-14d+7}{4\left(d-1\right)^{2}}+\frac{R}{2\left(d-1\right)^{3}}+\frac{\left(2d^{2}-4d+1\right)R^{2}}{4\left(d-1\right)^{4}d^{2}} 
 & =\frac{\left(6+R\right)\left(2+R\right)}{16}+O\left(\frac{\dmax^4\, x}{n}\right).
\label{medium}
\end{align}
Finally, the error term from Theorem~\ref{main-sparse} is at most
\[ O\left(\frac{\dmax^4}{(d-2)n} + (d-2)\dmax\right) = O\left(\frac{\dmax^4}{x}
  \right).
\] 
Since this error term dominates the error from (\ref{medium}) under our
assumptions, the result follows.
\end{proof}

From Corollary~\ref{close-to-2} we see that when the average degree is
close to but above~2, and the geometric mean $\hat{d}$ is strictly greater than 2,
then $\E \tau_{\dvec}$ tends to infinity.
This can be true even for degree sequences where the probability of
connectivity tends to zero, even when Corollary~\ref{close-to-2} does not
apply. For example, consider the degree sequence
$\dvec$ with $n/2$ vertices of degree 5 and $n/2$ vertices of degree 1
(restricted to even $n$).  Here $d = 3$ and $\hat{d} = \sqrt{5} > 2$.
From Theorem~\ref{main-sparse} it follows that the expected number
of spanning trees in $\calG_{\dvec}$ is
\[ \Theta(1/n)\, \left(\frac{80}{27}\right)^{n/2}\]
which tends to infinity as $n\to \infty$.  However, the probability
that $\calG_{\dvec}$ is connected tends to zero.  To see this, we work
in the configuration model~\cite{bollobas}. 
For ease of notation, write $n=2t$ and let $S$ be the set
of configurations with $t$ cells containing 5 points and $t$ cells
containing 1 point.
If a configuration in $S$ gives rise to a connected graph 
then every point in a cell of size 1
is paired with a point from a cell of size 5. There are at most
\begin{equation}\label{example-numerator}
\frac{(5t)_t \, (4t)!}{2^{2t}\, (2t)!}
\end{equation}
such configurations,  and the probability that a random configuration
in $S$ is simple, conditioned on connectedness, is at most 1.
The total number of simple configurations in $S$ is
\begin{equation}\label{example-denominator}
\Theta(1) \, \frac{(6t)!}{2^{3t}\, (3t)!}
\end{equation}
where the $\Theta(1)$ factor is the probability that a random configuration in
$S$ is simple: this tends to a constant bounded away from zero, 
by~\cite[Theorem 4.6]{symm}.
Dividing \eqref{example-numerator} by \eqref{example-denominator} gives the upper bound 
\[ \Theta(1)\, \left( \frac{5^5}{2^7\, 3^3}\right)^{n/2} = o(1)\]
on the probability
that a random element of $\calG_{\dvec}$ is connected.

The case of dense irregular degree sequences will be treated
in a separate paper. 

\subsection{Outline of our approach }\label{s:outline}

Let $(a)_k$ denote the falling factorial $a(a-1)\cdots (a-k+1)$.  We say that a sequence $\xvec=(x_1,\ldots, x_n)$ of positive
integers  is a \emph{tree degree sequence} if the entries of $\xvec$ sum to $2n-2$.
We say that a tree degree sequence $\xvec$ is a
\emph{suitable} degree sequence if 
$1\leq x_j\leq d_j$ for all $j\in \{1,2,\ldots, n\}$.  (The intended meaning is that $\xvec$ is suitable
as a degree sequence for a spanning tree of a graph with degree sequence $\dvec$.)

For a suitable degree sequence $\xvec$, let $\calT_{\xvec}$ be the set of all 
trees with degree sequence~$\xvec = (x_1,\ldots, x_n)$
and $\calT$ be the set of all trees with vertex set $\{ 1,2,\ldots, n\}$.
It is well-known that
\begin{equation}\label{alltrees}
     \card{\calT_\xvec}= \binom{n-2}{x_1-1,\ldots,x_n-1}.
\end{equation}
(See for example~\cite[Theorem~3.1]{moon}.)
Let $\tau_{\dvec}(\xvec)$ denote the number of spanning trees of $\calG_{\dvec}$ with degree sequence $\xvec$, and denote by $P(\dvec,T)$ the probability that the 
random graph $\mathcal{G}_{\dvec}$  
has $T$ as a subgraph, for all $T\in \calT$.
Then the expected number of spanning trees with 
degree sequence $\xvec$ in $\mathcal{G}_{\dvec}$ can be written as
\begin{equation}
\label{nd}
 \E \tau_\dvec(\xvec) = \sum_{T\in \calT_\xvec} P(\dvec,T)
\end{equation}
and furthermore, the expected number of spanning trees (of any degree sequence)
in $\calG_{\dvec}$ is
\[ 
 \E \tau_\dvec = \sum_{\xvec} \E \tau_\dvec(\xvec)
\]
where the sum is over all suitable degree sequences $\xvec$.

We will estimate the summand in (\ref{nd}) using a theorem by 
McKay~\cite[Theorem~4.6]{symm}, which we will restate below,  including 
some necessary terminology, and with some minor rewording for consistency. 

\begin{thm}
\emph{\cite[Theorem 4.6]{symm}}\,\, 
\label{mckay}
Let $\gvec = (g_1,\ldots, g_n)$ be a sequence of non-negative integers with even sum $2m$,
and let $\gmax = \max\{ g_1,\ldots, g_n\}$. Let $X$ be a simple graph on the
vertex set $\{1,2,\ldots, n\}$ with degree
sequence $\xvec = (x_1,\ldots, x_n)$, where $\xmax = \max\{ x_1,\ldots, x_n\}$.
Suppose that $\gmax\geq 1$ and $\hat{\Delta} \leq \epsilon_1 m$,
where $\epsilon_1 < 2/3$ and $\hat{\Delta} = 2 + \gmax(\nfrac{3}{2}\gmax + \xmax + 1)$.
Define
\[ \lambda = \frac{1}{4m}\, \sum_{j=1}^n (g_i)_2 \quad \text{ and } \quad
      \mu = \frac{1}{2m}\, \sum_{ij\in X} g_ig_j.\]
Let $N(\gvec,X)$ denote the number of simple graphs with degree sequence $\gvec$ and no edge in common with
$X$.  Then 
\[ N(\gvec,X) = \frac{(2m)!}{m!\, 2^{m}\, \prod_{j=1}^n g_i!}\, \exp(-\lambda - \lambda^2 - \mu
   + O(\hat{\Delta}^2/m))
\] 
uniformly as $n\to\infty$. 
\end{thm}

Given a suitable degree sequence $\xvec$ and a tree $T\in\mathcal{T}_{\xvec}$, 
define the parameters
\begin{align*}
  \lambda_0 &= \frac{1}{2dn}\, \sum_{j=1}^n \, (d_j)_2, \\
  \lambda(\xvec) &= \frac{1}{2(d-2)n+4} \, \sum_{j=1}^n \, (d_j-x_j)_2, \\
  \mu(T) &= \frac{1}{(d-2)n+2}\, \sum_{\{i,j\}\in E(T)} (d_i-x_i)(d_j-x_j)
\end{align*}
Using Theorem~\ref{mckay}, we may prove the following.

\begin{lemma} 
\label{prob}
Suppose that $\xvec$ is a suitable degree sequence and let $T\in\mathcal{T}_{\xvec}$. 
With notation as above, provided $\dmax^4\leq (d-2)n$,
\begin{align*}
     P(\dvec,T) 
      &= \frac{ (dn/2)_{n-1}\, 2^{n-1}}{  (dn)_{2n-2} }
            \prod_{j=1}^n (d_j)_{x_j} \exp\left( \lambda_0 + \lambda_0^2 - \lambda(\xvec) - \lambda(\xvec)^2 - \mu(T)
           + O\!\left(\frac{\dmax^4}{(d-2)n}\right) \right). 
\end{align*}
\end{lemma}

\begin{proof}
There is a bijection between the set of graphs with degree sequence $\dvec$
which contain $T$, and those with degree sequence $\dvec - \xvec$ which contain
no edges of $T$.  Therefore, we can write
\begin{equation}
 P(\dvec,T) = \frac{N(\dvec-\xvec,T)}{N(\dvec,\emptyset)}
\label{PdT}
\end{equation}
and Theorem~\ref{mckay} to estimate the numerator and denominator.

First, consider $N(\dvec-\xvec, T)$. 
Let
\[ \hat{\Delta} = 2 + g_{\mx}\left(\nfrac{3}{2} g_{\mx} + x_{\mx} + 1\right)\]
where $g_{\mx} = \max_{j=1,\ldots, n} (d_j-x_j)$. 
We require that $\hat{\Delta}^2 \leq \eps_1 m$,
where $\eps_1 < \nfrac{2}{3}$ is a constant and $m = \nfrac{1}{2} ((d-2)n+2)$
is the number of edges in a graph with degree sequence $\dvec-\xvec$.
By assumption,
\[ m = \nfrac{1}{2}\, ((d-2)n+2) \geq 1 + \nfrac{1}{2}\dmax^4.\]
Since $g_{\mx}, x_{\mx}\leq d_{\mx}$ and $\dmax \geq 3$ (which follows since $d > 2$),
we have
\[ \frac{\hat{\Delta}}{m} \leq 
    \frac{2 + \dmax(\nfrac{5}{2}\dmax + 1)}{1 + \nfrac{1}{2} \, \dmax^4} \leq \frac{55}{83}
\]
which is strictly less than $\nfrac{2}{3}$.
Observe also that $\hat{\Delta} = O(\dmax^2)$.
Hence Theorem~\ref{mckay} applies and says that
\begin{align*}
N(\dvec-\xvec,T) &= \frac{((d-2)n+2)!}{((d-2)n/2 + 1)! 2^{(d-2)n/2+1}\, \prod_{j=1}^n (d_j-x_j)!}\\
  & \hspace*{4cm}
   \times \exp\left( -\lambda(\xvec) - \lambda(\xvec)^2 - \mu(T) + O\Big(\dmax^4/((d-2)n)\Big)\right).
\end{align*}
Similarly, we obtain 
\begin{align}
N(\dvec,\emptyset) &= \frac{(dn)!}{(dn/2)! 2^{dn/2}\, \prod_{j=1}^n d_j!}
           \, \exp\left( -\lambda_0 - \lambda_0^2 + O\Big(\dmax^4/((d-2)n)\Big)\right),
\label{Nd}
\end{align}
noting that the value of the $\hat{\Delta}$ is smaller than in the previous application
of Theorem~\ref{mckay}, while the parameter $m$ is larger.
Substituting these expressions into (\ref{PdT}) completes the proof.
\end{proof}

Observe that the only term in the argument of the exponential
in Lemma~\ref{prob} which depends on the structure of $T$
(rather than just the degree sequence of $T$) is $\mu(T)$.
For any suitable degree sequence $\xvec$ and any tree $T\in\calT_{\xvec}$, define
\begin{equation}
 f(\xvec)=\lambda_0+\lambda_0^2 - \lambda(\xvec)-\lambda(\xvec)^2 
\label{fg-def}
\end{equation}
and let 
\begin{equation}
\label{beta-def}
 \eg(\xvec) = \frac{1}{|\calT_\xvec|}\, \sum_{T\in\calT_{\xvec}}\, e^{- \mu(T)}
\end{equation}
be the average value of $e^{-\mu(T)}$ over all $T\in\calT_\xvec$. 

Combining (\ref{alltrees}), (\ref{nd}) and Lemma~\ref{prob}, 
for any suitable degree sequence $\xvec$ we have
\begin{align}
   \E \tau_\dvec(\xvec) &= e^{O(\dmax^4/((d-2)n))}
      \frac{ (dn/2)_{n-1} \,2^{n-1}\, \prod_{j=1}^n d_j}
                  {  (dn)_{2n-2} }\,
            \sum_{T\in\calT_\xvec}\, \biggl(\, \prod_{j=1}^n (d_j-1)_{x_j-1}
                e^{f(\xvec) - \mu(T)} \biggr) \nonumber \\
       &= e^{O(\dmax^4/((d-2)n))}
      \frac{ (dn/2)_{n-1} \,2^{n-1}\, \hat{d}^n} 
                  {  (dn)_{2n-2} } \, (n-2)!\,
            \biggl(\, \prod_{j=1}^n \binom{d_j-1}{x_j-1} \biggr)
                  e^{f(\xvec)}\, \eg(\xvec). \label{one-x}
\end{align}
Now define 
\begin{equation}
\label{mubar-def}
 \barg(\xvec) = \frac{1}{|\calT_{\xvec}|}\, \sum_{T\in\calT_{\xvec}} \mu(T),
\end{equation} 
the average value of $\mu(T)$ over $\calT_{\xvec}$.
By proving that
$\eg(\xvec)$ is close to $e^{-\barg(\xvec)}$ for each suitable degree sequence 
$\xvec$, and evaluating $\barg(\xvec)$, we will establish the following.

\newpage   

\begin{thm}
Suppose that the conditions of Theorem~\ref{main-sparse} hold and that 
$\xvec$ is a suitable degree sequence.  
Then with $\eta$, $R$ and $\Hd$ as defined as in Theorem~\ref{main-sparse},
\begin{align*}
 \E \tau_{\dvec}(\xvec) 
= \Hd\, \binom{(d-1)n}{n-2}^{\!-1}
&\biggl(\,\prod_{j=1}^n \binom{d_j-1}{x_j-1} \biggr)
  \,\exp\biggl( 
\frac{(R+d^2)^2}{4d^2} - \dfrac{1}{4} - \lambda(\xvec) - \lambda(\xvec)^2 \\
&{\kern5em}- \frac{1}{n}\sum_{j=1}^n (x_j-1)(d_j-x_j)
    + O\left(\frac{\dmax^4}{(d-2) n} + \eta \right)\biggr).
      \end{align*}
      \label{just-one-x}
\end{thm}

The structure of the rest of the paper is as follows.
In Section~\ref{s:tree-averaging} we generalise the function~$\mu$
and prove a concentration result for trees with
given degrees.  This proof will involve a martingale concentration
result of McDiarmid~\cite{McDiarmidChapter} which we discuss in Section~\ref{s:concentration}.
The results of Section~\ref{s:tree-averaging} are applied in
Section~\ref{s:completing-calculations} to prove that the average of $e^{-\mu(T)}$
over $T\in\calT_{\xvec}$ is close to $\exp(-\barg(\xvec))$, and hence to prove
Theorem~\ref{just-one-x},
for any suitable degree sequence $\xvec$.
Finally, Theorem~\ref{main-sparse} is proved in Section~\ref{s:final-approx}.

Before we begin, note that we use the following conventions in our summation notation: $\sum_{i\neq j}$ will always denote a sum over all ordered pairs $(i,j)$ with $i\neq j$ (over some appropriate range which will be clear from the context, usually $i,j = 1,\ldots, n$).
On the other hand, if $i$ is fixed and we wish to sum over all $j\neq i$
(for example, over all $j\in\{ 1,2,\ldots, n\}\setminus i$) then we will write
$\sum_{j: j\neq i}$.  

\section{Concentration results}\label{s:concentration}

Let  $\mathcal{P}=(\varOmega,\calF,\mathbb{P})$ be a finite probability space.
A sequence $\calF_0,\ldots,\calF_n$ of $\sigma$-subfields
of $\calF$ is a \textit{filter} if $\calF_0\subseteq\cdots\subseteq\calF_n$.   A sequence  $Y_0,\ldots,Y_n$  
of random variables on $\mathcal{P}$ is a
\textit{martingale with respect to $\calF_0,\ldots,\calF_n$} if
\begin{itemize}\itemsep=0pt
\item[(i)] $Y_j$ is $\calF_j$-measurable and has finite expectation, 
for $j=0,\ldots, n$;
\item[(ii)] $\E(Y_j \st \calF_{j-1}) = Y_{j-1}$ for $j=1,\ldots, n$.
\end{itemize}
An important example of a martingale is made by the so-called
\textit{Doob martingale process}.
Suppose $X_1,X_2,\ldots,X_n$ are random variables on $\mathcal{P}$ and
$f(X_1,X_2,\ldots,X_n)$ is a random variable on $\mathcal{P}$ of bounded expectation.
Let $\sigma(X_1,\ldots, X_j)$ denote the $\sigma$-field generated
by the random variables $X_1,\ldots, X_j$. 
Define the martingale 
$\{Y_j\}$ with respect to the filter $\{\calF_j\}$, where for each~$j$,
$\calF_j=\sigma(X_1,\ldots,X_j)$ and
$Y_j=\E(f(X_1,X_2,\ldots,X_n)\st \calF_j)$.
In particular, $\calF_0=\{\emptyset,\varOmega\}$ and
$Y_0=\E f(X_1,X_2,\ldots,X_n)$.

In this section we state some concentration results for martingales.
See McDiarmid~\cite{McDiarmidChapter} for further background and for any definitions not 
given here.
Following McDiarmid~\cite{McDiarmidChapter}, for $j=1,\ldots, n$ 
we define the \emph{conditional range} of $Y_j$ as
\begin{equation} 
  \operatorname{ran}(Y_j\mid \calF_{j-1}) = 
  \operatorname{ess\, sup}(Y_j\mid \calF_{j-1}) +
  \operatorname{ess\, sup}(- Y_j\mid \calF_{j-1}).
\label{con-range}
\end{equation}
Here ``essential supremum'' may be replaced by ``supremum'', as in~\cite{McDiarmidChapter}, if the
probability distribution is positive over $\varOmega$.

Our main tool is the following result from McDiarmid~\cite{McDiarmidChapter}.
The tail bound on the probability is given by~\cite[Theorem 3.14]{McDiarmidChapter}.
The upper estimate on the moment generating function $\E(e^{h Y_n})$
is an intermediate step of McDiarmid's proof, see~\cite[Section 3.5]{McDiarmidChapter}.
The lower bound on $K$ is due to Jensen's inequality. 

\begin{thm} \emph{(\cite{McDiarmidChapter})}\
Suppose that $\mathcal{P}=(\varOmega,\calF,\mathbb{P})$ is a finite probability space. 
Let $Y_0,Y_1,\ldots, Y_n$ be a martingale  on 
 $\mathcal{P}$ with respect to a filter 
$\calF_0,\calF_1\ldots,\calF_n$, 
where $\calF_0 = \{\emptyset,\varOmega\}$,
such that
\[ \sum_{j=1}^n \left( \operatorname{ran}(Y_j\mid \calF_{j-1})\right)^2
                            \leq \hat{r}^2\quad a.s.\
\]
for some real $\hat{r}$.  Then
\[ \E e^{Y_n}  = e^{Y_0 + K}
\]
where $0\leq K\leq \nfrac{1}{8}\, \hat{r}^2$.
Furthermore, for any real $t > 0$,
\[ \Pr(|Y_n - Y_0| \geq t) \leq 2\, \exp( -2t^2/\hat{r}^2).\]
\label{martingales}
\end{thm}

As a corollary, we obtain a concentration result for functions of sets of a given
size.

\begin{cor}\label{subsets}
Let $\binom{[N]}{r}$ be the set of $r$-subsets of
$\{1,\ldots,N\}$ and let $h:\binom{[N]}{r}\to \Reals$ be given.
Let $C$ be a uniformly random element of $\binom{[N]}{r}$.
Suppose that there exists $\alpha\geq 0$ such that
\[
 \abs{h(A)-h(A')} \le \alpha
\]
for any $A, A' \in \binom{[N]}{r}$ with $|A \cap A'| = r-1$.
Then
\begin{equation}
  \E e^{h(C)} = \exp\left(\E h(C)+ K\right)
\label{henry}
\end{equation}
where $K$ is a real constant such that $0 \leq K \leq \tfrac18 \min\{r,N-r\} \alpha^2 $.
Furthermore, for any real $t>0$,
\begin{equation*}
{\rm Pr}( |h(C) - \E h(C)| \geq t) \leq 2 \exp\left( -\frac{2t^2}{\min\{r,N-r\} \alpha^2}\right). 
\end{equation*}
\end{cor}

\begin{proof}
Let $S_N$ denote the set of permutations of $\{1,\ldots, N\}$
and $\tau = (\tau_1,\ldots, \tau_N)$ be a uniform random element of $S_N$.  
Note that the set $\{\tau_{1}, \ldots, \tau_r\}$ is a uniformly random 
element of $\binom{[N]}{r}$. 
Define $\tilde{h}:S_N \to \mathbb{R}$ by
$\tilde{h}(\omega) = h(\{\omega_{1}, \ldots, \omega_r\})$
for all $\omega\in S_N$.  
Then
\[  |\tilde{h}(\rho) - \tilde{h}(\rho')| \leq \alpha\]
for all permutations $\rho$, $\rho'\in S_N$ such that
$\rho^{-1}\rho'$ is a transposition.  
Given $\omega = (\omega_1,\ldots,\omega_N) \in S_N$, for $k=0,\ldots, N$
let
\[ \tilde{h}_k(\omega) = 
    \E\left(\tilde{h}(\tau) \st \tau_j=\omega_j \, \text{ for } \, j=1,\ldots k
  \right).\]
Clearly, $\tilde{h}_0(\tau), \ldots, \tilde{h}_N(\tau)$ forms a martingale: it is the result of the Doob martingale process for $\tilde{h}(\tau)$.
It follows from Frieze and Pittel~\cite[Lemma 11]{FP} that
\[ \operatorname{ran}\(\tilde{h}_k(\tau) \mid \sigma(\tau_1,\ldots,   \tau_{k-1}) \) \leq \alpha.
\]
Moreover, for any $\omega \in S_N$ and $k\in\{ r,\ldots, N\}$, we have 
 $$
 	\E\(\tilde{h}(\tau) \st \tau_j=\omega_j, \ 1\leq j\leq k\) = h(\{\omega_{1}, \ldots, \omega_r\}).
 $$ 
Therefore 
$\operatorname{ran}\(\tilde{h}_k(\tau) \mid \sigma(\tau_1,\ldots,   \tau_{k-1}) \)  = 0$
for all $k > r$.
Applying Theorem~\ref{martingales} to the martingale $\tilde{h}_0(\tau),\ldots, \tilde{h}_N(\tau)$, 
we conclude that (\ref{henry}) holds with 
$0\leq K \leq \tfrac18 r \alpha^2$,  and that
\[ {\rm Pr}( |h(C) - \E h(C)| \geq t) \leq 
  2 \exp\left( -\frac{2 t^2}{r \alpha^2}\right).\] 
If $r\leq N-r$ then we are finished. Otherwise we repeat the above argument using the 
bijection between subsets and their complements.   
\end{proof}

\section{Trees with given degrees}\label{s:tree-averaging}

In this section we consider sums of the form
\begin{equation}
\label{Fdef}
 F(T) = \sum_{\{ j,k\} \in E(T)}\phi(j)\, \phi(k)
\end{equation}
for a given function $\phi:\{ 1,2,\ldots, n\}\rightarrow [\a, \b]\subset \mathbb{R}$. 

Let $\bar{F}(\xvec)$ be the average value of $F$ over all trees with
a given degree sequence $\xvec$:
\[ \bar{F}(\xvec) = \frac{1}{|\calT_{\xvec}|}\, \sum_{T\in\calT_{\xvec}} F(T).\]
The goal of this section is to prove the following theorem, showing that the 
average of $e^{\xi  F(T)}$ over $\calT_{\xvec}$
is close to $e^{\xi \bar{F}(\xvec)}$, for $\xi\in \{ -1,1\}$. 
We will measure this distance using the seminorm of $\phi$ given by
\begin{equation}
\label{phi-norm}
 \| \phi\|_m = \min_{c\in \mathbb{R}}\, \sum_{j=1}^n |\phi(j) - c |.
\end{equation}
Here the minimising value of $c$ is any median of $\{ \phi(j) : j = 1,\ldots, n\}$.

\begin{thm} Let $F$ satisfy \emph{(\ref{Fdef})}.  Then for any tree degree sequence $\xvec$
and for $\xi\in\{-1,1\}$,
\[
\frac{1}{|\calT_{\xvec}|}\, \sum_{T\in\calT_{\xvec}}\, e^{\xi F(T)}  =\exp\left(\xi \bar{F}(\xvec)+ K\right)
\]
for some real constant $K$ which satisfies 
$0\leq K\leq \nfrac{1}{8} L_\phi$, where
\[ 
  L_\phi =  
    (\b - \a)^3\, \min\big\{ (\b- \a) n,\,\,  \| \phi\|_m\, (\ln n + 2)\big\}.
\]
Furthermore, if $\widehat{T}$ is a uniformly random element of $\calT_{\xvec}$
then for any real constant $t>0$,
\[ \Pr( |F(\widehat{T}) - \bar{F}(\xvec)| > t) 
            \leq 2\,\exp\left( - 2 t^2/L_\phi \right). \]
\label{tree-general}
\end{thm}

First we give some explicit formulae which we will need later.

\begin{lemma}\label{tree-prob}
 Let $\xvec$ be a tree degree sequence and consider the set
 $\cal{T}_\xvec$ of all trees with degree sequence~$\xvec$.
 \begin{itemize}
\item[\emph{(i)}]
Let $S$ be a disconnected forest with vertex set $\{1,\ldots,n\}$ and
degree sequence $(s_1,\ldots,s_n)$, where $s_j\le x_j$ for $j = 1,\ldots, n$.
Let $S_1,\ldots,S_r$ be the components of~$S$.
Then the probability that a uniform random tree in $\calT_\xvec$
contains $S$ is
\[
 \frac{ \prod_{i=1}^r \sum_{j\in V(S_i)} (x_j-s_j)}{(n-2)_{n-r}}
 \prod_{j=1}^n \, (x_j-1)_{s_j-1},
\]
where $(x_j-1)_{s_j-1}=x_j^{-1}$ if $s_j=0$.
In particular, for distinct $j,k\in \{1,2,\ldots, n\}$, the fraction of trees in $\calT_\xvec$ in which vertices $j,k$ are adjacent is
 \[ 
\frac{x_j+x_k-2}{n-2}.\]
\item[\emph{(ii)}] The average value of $F$ over $\calT_{\xvec}$ is
\[ \bar F(\xvec) = \frac{1}{n-2}\, \left(\sum_{k=1}^n\phi(k)\right)\left(\sum_{j=1}^n (x_j-1)\, \phi(j)\right) -
  \frac{1}{n-2}\, \left(\sum_{j=1}^n (x_j-1)\, \phi(j)^2\right).
\]
\end{itemize}
\end{lemma}

\begin{proof}
Define $\xvec' = (x'_1,\ldots,x'_r)$, where $x'_i = \sum_{j\in V(S_i)} (x_j - s_j)$
for $i=1,\ldots, r$.
If $x'_i=0$ for any $i$ then 
$T$ cannot contain $S$, as $S$ is disconnected. 
Hence the result holds trivially in that case, and for the remainder of the proof
we may assume that all entries of $\xvec'$ are positive.
Next, observe that the entries of $\xvec'$ sum to $2(r-1)$, and  hence $\xvec'$ is a tree
degree sequence. 

Each tree in $\calT_\xvec$ that contains $S$ can be formed uniquely by the following process:
\begin{enumerate}\itemsep=0pt
\item[(1)]
Take any tree $T'$ on the vertex set $\{1,\ldots,r\}$ with degree sequence
$\xvec'$.
\item[(2)] For $i=1,\ldots, r$, replace vertex $i$ of $T'$ by $S_i$ and distribute
the edges of $T'$ that were incident with~$i$
amongst the vertices of $S_i$, so that each vertex $j\in V(S_i)$ has degree $x_j$ in the resulting tree.
\end{enumerate}
By (\ref{alltrees}), the number of choices for $T'$ in Step~1 is
$\binom{r-2}{x'_1-1,\ldots,x'_r-1}$,
while the number of ways to distribute edges in Step~2 is
\[ \prod_{i=1}^r\, \frac{x'_i!}{\prod_{j\in V(S_i)} (x_j-s_j)!}.
\]
The first statement of (i) 
is proven by multiplying these expressions together, dividing by~\eqref{alltrees} and simplifying. 
Then taking $S$ to be the edge $jk$ together with $n-2$ trivial components
completes the proof of (i).

Now using linearity of expectation, (\ref{Fdef}), and part (i), we calculate that
\begin{align}
(n-2)\, \bar F(\xvec) &=  \sum_{j < k} \,(x_j+x_k-2)\, \phi(j)\phi(k) \label{fred0}\\
 &= \sum_{j\neq k} (x_j-1) \, \phi(j)\phi(k) \label{fred}\\
   &= \sum_{j=1}^n (x_j-1)\, \phi(j)\, \left(\left(\sum_{k=1}^n \phi(k)\right) - \phi(j)\right) \nonumber \\
       &=  \left(\left(\sum_{k=1}^n \phi(k)\right)\, \sum_{j=1}^n (x_j-1)\phi(j) \right)
          -  \left(\sum_{j=1}^n (x_j-1)\,\phi(j)^2\right),\nonumber
   \end{align}
establishing (ii).  
\end{proof}

We complete this section with the proof of Theorem~\ref{tree-general}, which
involves the process used to construct the \emph{Pr{\" u}fer code} of a labelled tree.
The Pr{\" u}fer code of a tree $T\in\calT$ is a sequence
$\boldsymbol{b} = (b_1,\ldots, b_{n-2})\in \{1,2,\ldots, n\}^{n-2}$.
Given~$T$, find the unique neighbour $b_1$ of the lowest-labelled leaf $a_1$. 
Then $b_1$ becomes the first entry in the Pr{\" u}fer code for $T$. 
We find the next entry recursively
by considering the tree $T - a_1$ with the first leaf deleted.  The process stops
when a single edge remains: this edge is determined by the
degree sequence and does not need to be recorded in the code $\boldsymbol{b}$.
We will refer to this process as the \emph{Pr{\" u}fer process} with input $T$. 
See Figure~\ref{f:prufer} for an example.
\begin{figure}[ht!]
\begin{center}
\begin{tikzpicture}
\draw [-] (0,1) -- (4,1);
\draw [-] (1,1) -- (1,2);
\draw [-] (2,1) -- (2,2);
\draw [fill] (0,1) circle (0.1);
\draw [fill] (1,1) circle (0.1);
\draw [fill] (1,2) circle (0.1);
\draw [fill] (2,1) circle (0.1);
\draw [fill] (2,2) circle (0.1);
\draw [fill] (3,1) circle (0.1);
\draw [fill] (4,1) circle (0.1);
\node [below] at (0,0.9) {3};
\node [below] at (1,0.9) {2};
\node [left] at (0.9,2) {6};
\node [below] at (2,0.9) {7};
\node [right] at (2.1,2) {4};
\node [below] at (3,0.9) {1};
\node [below] at (4,0.9) {5};
\node at (6,1) {\huge $\Rightarrow$};
\node at (8,1) {$(2,7,1,7,2)$};
\end{tikzpicture}
\caption{A tree and its corresponding Pr{\" u}fer code.}
\label{f:prufer}
\end{center}
\end{figure}
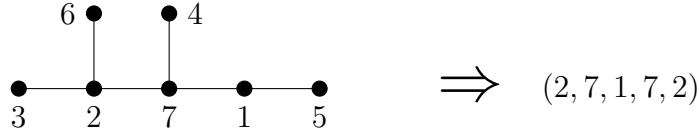
The correspondence between trees and Pr{\" u}fer codes is a bijection:
see for example Moon~\cite[pp.~5-6]{moon}. 
This provides a proof of Cayley's formula and of (\ref{alltrees}).

The following useful property of the Pr{\" u}fer process may be proved by induction on $j$.

\begin{lemma}
Let $\xvec$ be a tree degree sequence and let $T\in\calT_{\xvec}$. 
Suppose that the Pr{\" u}fer process with input $T$ produces
the Pr{\" u}fer code $\bvec$ and the sequence $(a_1,\ldots, a_{n-2})$
of ``leaves''.  For any $j = 1,\ldots, n-2$,
the initial sequence $(a_1,\ldots, a_j)$ is uniquely determined by $\xvec$
and $(b_1,\ldots, b_{j-1})$.
\label{determined}
\end{lemma}

When there is more than one tree under consideration we will write $a_j(T)$, 
$b_j(T)$ for the vertices identified at step $j$ of the Pr{\" u}fer process
for the tree $T$.  
To prove Theorem~\ref{tree-general} we work with a martingale defined using the Pr{\" u}fer code of a tree.  
A martingale construction based on the Pr{\" u}fer code was given by Cooper, McGrae and Zito~\cite{CMZ}, 
for all labelled trees.  Our martingale is restricted to trees with a given degree sequence and we study a
function for which it is more difficult to bound the conditional ranges.

\begin{proof}[Proof of Theorem~\ref{tree-general}.]\
Suppose that $T_1$ and $T_2$ are trees on $\{ 1,2,\ldots, n\}$ with the same degree sequence.
For $j=0,\ldots, n-3$, say that $T_1$ and $T_2$ are \emph{$j$-equivalent},
and write $T_1 \simj{j} T_2$, if $b_i(T_1) = b_i(T_2)$ for $i=1,2,\ldots, j$.
By Lemma~\ref{determined}, if $T_1\simj{j} T_2$ then
$a_i(T_1) = a_i(T_2)$ for $i=1,\ldots, j$. 
The $j$-equivalence relation induces a partition of $\calT_{\xvec}$ into
equivalence classes $C_{j,1},\ldots, C_{j,r_j}$, say, with
$\cup_{\ell=1}^{r_j} C_{j,\ell} = \calT_{\xvec}$.

Let $Y_{j,\ell}$ equal the average of $F(T)$ over $T\in C_{j,\ell}$, and define the
function $Y_j$ on $\calT_{\xvec}$ by $Y_j(T) = Y_{j,\ell}$ if $T\in C_{j,\ell}$.
Finally, define the random variable $Y_j = Y_j(\widehat{T})$,
where $\widehat{T}$ is a uniformly random element of $\calT_{\xvec}$.
Then $Y_0$ is the constant function which takes the value $\E F(\widehat{T})$ 
everywhere, and $Y_{n-3} = F(\widehat{T})$, since each equivalence class $C_{n-3,\ell}$ 
is a set of size 1.
Observe that $Y_0,\ldots, Y_{n-3}$ is a martingale with respect to the filter $\calF_0, \ldots, \calF_{n-3}$, 
where, for each $j$,   $\calF_j$ is generated by the sets $C_{j,1},\ldots, C_{j,r_j}$. In fact,
 this is the Doob martingale process for the function $F(\widehat{T})$ of the 
random variables $b_1(\widehat{T}),\ldots,b_{n-3}(\widehat{T})$, which 
determine $\widehat{T}$ uniquely.

To apply Theorem~\ref{martingales} we must calculate a value for $\hat{r}^2$.
Suppose that $T_1$ and $T_2$ are $(j-1)$-equivalent, where $T_1, T_2\in\calT_{\xvec}$ and $j\in \{1, \ldots, n-3\}$.
Then $a_j(T_1) = a_j(T_2)$, again by Lemma~\ref{determined}.  
For ease of notation, write
$a_i$ instead of $a_i(T_1)$ (or $a_i(T_2)$) for $i=1,\ldots, j$, and write
$b_i$ instead of $b_i(T_1)$
(or $b_i(T_2)$) for $i=1,\ldots, j-1$. 

For $s=1,2$ let $T_s'$ 
be the tree (with $n-j$ vertices) obtained by deleting the vertices
$a_1,\ldots, a_j$ from $T_s$.    Both $T'_1$ and $T'_2$ have vertex set 
\[ V_j = \{ 1,2, \ldots, n\} \setminus \{ a_1,\ldots, a_{j}\}.\]
If
$b_j(T_1) = b_j(T_2)$ then $T'_1$ and $T'_2$ have the same degree sequence, 
since (in this case) precisely the same edges have been deleted from $T_1$ and~$T_2$.
In this case, $Y_j(T_1)= Y_j(T_2)$.

Otherwise, the degree sequences of $T'_1$, $T'_2$ differ only for the two
vertices $b_j(T_1)$ and $b_j(T_2)$.
Specifically, vertex $b_j(T_1)$ has degree in $T_1'$ which is equal to its
degree in $T_2'$ minus 1, while vertex $b_j(T_2)$ has degree in $T_1'$ which is equal
to its degree in $T_2'$ plus 1.  Hence $T'_1$ and $T'_2$ have the same
degree on all vertices in the set
\[ U_j(T_1,T_2) =  V_j \setminus \{ b_j(T_1),\, b_j(T_2)\}.\]
For $s=1,2$, let $\yvec_s$ be the degree sequence of $T_s'$ (on the vertex set
$V_j$) and let $\calT'_s$ denote the set of all trees on the vertex set $V_j$
with degree sequence $\yvec_s$.  
Observe that $\yvec_s$ and $\calT'_s$ depend only on
$(b_1,\ldots, b_{j-1}, b_j(T_s))$ and $\xvec$.
By relabelling the equivalence classes if necessary, we may assume that $T_s\in C_{j,s}$ for $s=1,2$.
The map $\varphi:C_{j,s} \rightarrow \calT'_s$
which sends a tree $T\in C_{j,s}$ to $T\setminus \{ a_1,\ldots, a_j\}$
is a bijection.  To see this, observe that the inverse map $\varphi^{-1}$
takes a tree in $\calT'_s$,
adds the vertices $a_1,\ldots, a_j$ and the edges 
\[ \{ \{ a_1,b_1\},\ldots, \{ a_{j-1}, b_{j-1}\},\,\, \{ a_j, b_j(T_s)\}\}
\] giving a tree in $C_{j,s}$.  Therefore, for $s=1,2$,
\[ \frac{1}{|C_{j,s}|}\, \sum_{T\in C_{j,s}} F(T\setminus \{ a_1,\ldots, a_j\})
 = \frac{1}{|\calT'_s|}\, \sum_{T'\in \calT'_s} F(T').
\]
Combining this with (\ref{Fdef}) and the definition of $Y_{j,s}$, 
we see that for $s=1,2$,
\begin{align*}
 Y_{j,s} &= \frac{1}{|C_{j,s}|}\,\sum_{T\in C_{j,s}} 
  \, \sum_{\{k,\ell\}\in E(T)} \phi(k)\phi(\ell)\\
  &= \left( \sum_{i=1}^{j-1} \phi(a_i)\phi(b_i)\right)
                + \phi(a_j)\phi(b_j(T_s)) + 
    \frac{1}{|C_{j,s}|}\, \sum_{T\in C_{j,s}} F(T\setminus \{ a_1,\ldots, a_j\})\\
  &= \left( \sum_{i=1}^{j-1} \phi(a_i)\phi(b_i)\right)
                + \phi(a_j)\phi(b_j(T_s)) + 
    \frac{1}{|\calT'_s|}\, \sum_{T'\in \calT'_s} F(T').
\end{align*}
Applying Lemma~\ref{tree-prob}(ii) gives
\begin{align*}
Y_{j,1} - Y_{j,2}
  &= \left(\phi(b_j(T_1)) - \phi(b_j(T_2))\right)\,
     \biggl( \phi(a_j) - \frac{1}{n-j-2}\sum_{\ell \in U_j(T_1,T_2)} \phi(\ell) \biggr)\\
      &= \frac{\phi(b_j(T_1)) - \phi(b_j(T_2))}{n-j-2}\, \sum_{\ell\in U_j(T_1,T_2)}
        \(\phi(a_j) - \phi(\ell)\).
        \end{align*}
(Note that if $T_1\simj{j} T_2$ then $b_j(T_1)=b_j(T_2)$ and the above
equality also holds.)

Recall the definition of $\|\phi\|_m$  from (\ref{phi-norm}), 
and let $c\in\mathbb{R}$ be the minimising value in this definition. 
By the triangle inequality,
\begin{align}
 \frac{1}{n-j-2}\, \biggl|\sum_{\ell\in U_j(T_1,T_2)} (\phi(a_j) - \phi(\ell)) \biggr| 
   &\leq |\phi(a_j) - c| + \frac{1}{n-j-2}\, \biggl|\sum_{\ell\in U_j(T_1,T_2)} (c - \phi(\ell))\biggr| \nonumber \\
   &\leq |\phi(a_j) - c| + \frac{\| \phi\|_m }{n-j-2}
\label{intermediate}
   \end{align}
since $U_j(T_1,T_2)$ has $n-j-2$ elements and $j\leq n-3$. 
Therefore, for any equivalence class $C_{j-1,\ell}$, we have
\begin{align}
& \biggl(\, \sup_{T'\in C_{j-1,\ell}}  Y_j(T') + \sup_{T'\in C_{j-1,\ell}} (-Y_j(T'))\biggr)^2 \nonumber\\
 &\hspace*{3cm} \leq \frac{(\b-\a)^2}{(n-j-2)^2}\,
    \sup_{T_1,T_2\in C_{j-1,\ell}} \biggl(\, \sum_{\ell\in U_j(T_1,T_2)} (\phi(a_j) - \phi(\ell))\biggr)^{\!2}\nonumber\\
    &\hspace*{3cm} \leq (\b-\a)^3\, \min\left\{ \b-\a,\, | \phi(a_j) -c| + \frac{\| \phi\|_m}{n-j-2}\right\}.\label{OK}
    \end{align}
(Here we take the minimum of two possible upper bounds: the first arises from taking the
worst case summand for both factors in the line above, while the second arises by 
applying (\ref{intermediate}) to one of the factors.)

Now let $C_{j-1}(\widehat{T})$ denote the random set which is the equivalence
class with respect to $\simj{j-1}$ which contains $\widehat{T}$.
It follows from (\ref{OK}) that   
\begin{align*}		
\operatorname{ran}(Y_j\mid \calF_{j-1})^2 &=
  \biggl(\,\sup_{T\in C_{j-1}(\widehat{T})}  Y_j(T) 
         + \sup_{T\in C_{j-1}(\widehat{T})} (-Y_j(T)) \biggr)^2 
	\\
	&\leq  (\b-\a)^3\, \min\left\{ \b-\a,\, | \phi(a_j(\widehat{T})) -c| + \frac{\| \phi\|_m}{n-j-2}\right\}.
\end{align*}
Using  the definition of $c$, the standard upper bound on the harmonic series
and the fact that each vertex is chosen as $a_j(\widehat{T})$ at most once during 
the Pr{\" u}fer process, we get that 
\[ \sum_{j=1}^{n-3} \operatorname{ran}(Y_j\mid \calF_{j-1})^2
    \leq (\b - \a)^3 \, \min\left\{ (\b - \a) n,\,\,  \| \phi\|_m\ (\ln n + 2)\right\}.
\]
Observe that the left hand side does not change if $F$ is replaced by $-F$
(and hence, the same bound is obtained whether $\xi = 1$ or $\xi = -1$).
Since $\E(e^{Y_{n-3}}) =
\E(e^{F(\widehat{T})})$ and $Y_0 = \E(F(\widehat{T}))$,  applying
Theorem~\ref{martingales} completes the proof.
\end{proof}

\section{Proof of Theorem~\ref{just-one-x}}\label{s:completing-calculations}

First we note the following corollary of Theorem~\ref{tree-general}.
Recall the definition of $\eg(\xvec)$ and $\barg(\xvec)$ from (\ref{beta-def}), (\ref{mubar-def}),
respectively.  

\begin{lemma} Under the conditions of Theorem~\ref{main-sparse},
\[
\eg(\xvec) =\exp\left(-\barg(\xvec)+O\!\left(\min\left\{\frac{\dmax^4}{(d-2)^2n},\,
  \frac{\dmax^3\, \ln n}{(d-2)n}\right\} \right)\right).
\]
\label{Eeg}
\end{lemma}

\begin{proof}
Set
\[ \phi(j) = \frac{d_j-x_j}{\sqrt{(d-2)n+2}}\]
for $j\in \{ 1,2,\ldots, n\}$, and let $\xi=-1$.   
We can take $\a=0$ and $\b=\dmax/\sqrt{(d-2)n+2}$.  Next, we bound
\[ \| \phi\|_m \leq \sum_{j=1}^n \frac{d_j-x_j}{\sqrt{(d-2)n+2}} = \sqrt{(d-2)n + 2}.\]
Finally, observe that
\[ \frac{\dmax^3 (\| \phi\|_m + \b)(\ln n + 2)}{((d-2)n+2)^{3/2}}
  = O\left(\frac{\dmax^3\, \ln n}{(d-2)n} + \frac{\dmax^4\ln n}{((d-2)n)^{2}}\right)
    = O\left(\frac{\dmax^3\,\ln n}{(d-2)n}\right).
    \]
Now the result follows from Theorem~\ref{tree-general}.
\end{proof}

We can now prove Theorem~\ref{just-one-x}, giving an asymptotic
expression for the expected number $\E \tau_\dvec(\xvec)$ of spanning trees
in $\calG_{\dvec}$ with degree sequence $\xvec$.  

\begin{proof}[Proof of Theorem~\ref{just-one-x}]\
Firstly note that, by \eqref{fred},
\begin{align}
 \barg(\xvec) 
   &= \frac{1}{(n-2)((d-2)n+2)}\, 
    \sum_{j\ne k} \,(x_j-1)(d_j-x_j)(d_k-x_k)
  \label{g-useful}\\
  &=  \frac{1}{n}\, \sum_{j=1}^n (x_j-1)(d_j-x_j) 
               + O\left(\frac{\dmax^2}{(d-2)n}\right). \label{barg-expression}
\end{align}
We rewrite (\ref{one-x}) as
\begin{align}
&\E \tau_{\dvec}(\xvec) \nonumber\\*
&\quad = e^{O(\dmax^4/((d-2)n))}\,
  \frac{(dn/2)_{n-1}\, 2^{n-1}\, \hat{d}^n}{(dn)_{n}}\, 
         \binom{(d-1)n}{n-2}^{\!\!-1}\, \left(\prod_{j=1}^n \binom{d_j-1}{x_j-1} \right) \,
     e^{f(\xvec)}\, \beta(\xvec)
\label{new}
\end{align}
with $f(\xvec)$ as defined in (\ref{fg-def}).
Applying Stirling's approximation gives
\[ \frac{(dn/2)_{n-1}\, 2^{n-1}\, \hat{d}^n}{(dn)_{n}} = \Hd\,
  \left(1 + O\left(\frac{1}{(d-2)n}\right)\right)\]
where $\Hd$ is defined in the statement of Theorem~\ref{main-sparse}.
Combining Lemma~\ref{Eeg} and (\ref{barg-expression}) gives
\begin{align}
  \beta(\xvec) 
 &= 
   \exp\left(-\barg(\xvec) + O\left(\min\left\{\frac{\dmax^4}{(d-2)^2 n},\,
   \frac{\dmax^3\, \ln n}{(d-2)n}\right\} \right)\right)\nonumber \\[0.5ex]
  &= \exp\left(-
   \frac{1}{n}\, \sum_{j=1}^n (x_j-1)(d_j-x_j) \right. \nonumber \\[-0.6ex]
  & \left. \hspace*{3cm} {} 
    + O\left(\frac{\dmax^2}{(d-2)n} + \min\left\{\frac{\dmax^4}{(d-2)^2 n},\,
   \frac{\dmax^3\, \ln n}{(d-2)n}\right\} \right)\right). 
    \label{beta-approx}
\end{align}
In some cases, when $d-2$ is small,
we can obtain a smaller error bound by a different argument.
Observe that 
\begin{equation}
\label{beta-bounds}
  e^{- \barg(\xvec)} \leq \beta(\xvec)\leq 1,
\end{equation}
using Jensen's inequality for the lower bound.
It follows from (\ref{barg-expression}) that
\[ 
  \barg(\xvec) = O\left(\frac{\dmax^2}{(d-2)n} + (d-2)\dmax\right).
\]
Hence we can replace the upper bound on $\beta(\xvec)$ in (\ref{beta-bounds})  
by $e^{-\barg(\xvec)}$ if
we include an error term of this magnitude, leading to 
\begin{align} 
  \beta(\xvec) 
   &= \exp\left(-\barg(\xvec) 
    + O\left( \frac{\dmax^2}{(d-2)n}  + (d-2)\dmax\right)\right) \nonumber \\
   &= \exp\biggl(- \frac{1}{n}\, \sum_{j=1}^n (x_j-1)(d_j-x_j) 
           + O\biggl( \frac{\dmax^2}{(d-2)n} + (d-2)\dmax\biggr) \biggr)\label{beta-alternative}
\end{align}
using (\ref{barg-expression}).
We may choose to use either this expression or (\ref{beta-approx}), whichever
gives the smaller bound.  
Finally, observe that
\begin{equation}
\label{lambda0}
 \lambda_0 + \lambda_0^2 = \frac{(R+d^2)^2}{4d^2} - \dfrac{1}{4}.
\end{equation}
Combining this with (\ref{fg-def}), (\ref{new}), (\ref{beta-approx}) and 
(\ref{beta-alternative}) completes the proof.
\end{proof}

\section{Proof of Theorem~\ref{main-sparse}}\label{s:final-approx}

In this section we prove Theorem~\ref{main-sparse} by summing the expression from 
Theorem~\ref{just-one-x} over all suitable degree sequences $\xvec$.
Given a suitable degree sequence $\xvec$, define
\begin{equation} 
g(\xvec) = f(\xvec) - \barg(\xvec) = 
   \frac{(R+d^2)^2}{4d^2} - \dfrac{1}{4}  -\lambda(\xvec) - \lambda(\xvec)^2 - \barg(\xvec),
\label{g-def}
\end{equation}
using (\ref{lambda0}).
By (\ref{barg-expression}) and Theorem~\ref{just-one-x} we have
\begin{align}
\E \tau_\dvec &= 
 \Hd\, \sum_{\xvec} 
\binom{(d-1)n}{n-2}^{\!\!-1}\, \left(\,\prod_{j=1}^n \binom{d_j-1}{x_j-1}\right)
\nonumber \\ & \hspace*{4cm} {} \times
   \exp\left(  g(\xvec) 
    +  O\left( \frac{\dmax^4}{(d-2)n} + \eta \right)\right)\,
\label{sum-over-x}
\end{align}
where the sum is over all suitable degree sequences $\xvec$.
We now interpret this sum as an expected value of a function of a nonuniform 
distribution on suitable degree sequences. 

\begin{lemma}
\label{useful-later}
Fix a partition $A_1,\ldots, A_n$ of $\{ 1,2,\ldots, (d-1)n\}$
such that $\card{A_j} = d_j-1$ for $j=1,\ldots, n$, and let $B$ be
a uniformly random subset of $\{ 1,2,\ldots, (d-1)n\}$ of size $n-2$.  
Define the random
vector $\X = \X(B) = (X_1,\ldots, X_n)$ by $X_j = |A_j\cap B| + 1$.
Then
\[
     \E \tau_\dvec =  \Hd\, 
  \exp\left(  O\left( \frac{d_\mx^4}{(d-2)n} + \eta\right)\right)\,
  \E\left(e^{g(\X)}\right).
\]
\end{lemma}

\begin{proof}
Let $\xvec$ be a suitable degree sequence.
Since the sets $A_j$ are disjoint, there are $\prod_{j=1}^n \binom{d_j-1}{x_j-1}$
ways to choose a subset of $\{ 1,\ldots, (d-1)n\}$ with precisely
$x_j-1$ elements in $A_j$, for $j=1,\ldots, n$.  It follows that
%
\[ \Pr(\X = \xvec) 
   = \binom{(d-1)n}{n-2}^{\!\!-1}\, \prod_{j=1}^n \binom{d_j-1}{x_j-1}.\]
Substituting this into (\ref{sum-over-x}) completes the proof.
\end{proof}

Next, we prove that $\E\(e^{g(\X)}\)$ can be approximated by 
$e^{\E g(\X)}$ by applying Corollary~\ref{subsets}. 
We say that two suitable degree sequences $\xvec$ and $\xvec'$ are \emph{adjacent} if
$\xvec$ and $\xvec'$ differ in precisely two entries, say in entries $j$ and $k$,
such that $x'_j = x_j + 1$ and $x_k' = x_k-1$.
Adjacent degree sequences correspond to subsets $A, A'$ of $\{ 1,2,\ldots, (d-1)n\}$ of size
$n-2$ which have $n-3$ elements in common.
In order to apply Corollary~\ref{subsets} to $g$ we must bound the
amount by which $g(\xvec)$ can differ from $g(\xvec')$ when $\xvec$ and $\xvec'$
are adjacent.

\begin{lemma}
\label{f+g}
Suppose that $\xvec$, $\xvec'$ are two suitable degree
sequences which are adjacent.   Then
\[ \left| g(\xvec) - g(\xvec')\right| = 
    O\left(\frac{\dmax^2}{(d-2) n}\right).
\]
\end{lemma}

\begin{proof}
Recall the definition of $g$ in (\ref{g-def}).
Firstly, observe that
\[ \lambda(\xvec')^2 - \lambda(\xvec)^2 = 
  \left( \lambda(\xvec') - \lambda(\xvec)\right)\left(\lambda(\xvec') + \lambda(\xvec)\right)
    = O(\dmax)\, \left(\lambda(\xvec') - \lambda(\xvec)\right)
    \]
since for any suitable $\xvec$ we have
\[ \lambda(\xvec) = O\left(\frac{\dmax}{(d-2)n}\right)\, \sum_{j=1}^n (d_j-x_j) = O(\dmax).\]
Next we calculate that
\begin{align*}
|\lambda(\xvec') - \lambda(\xvec)| = \frac{| (d_k-x_k) -(d_j-x_j-1)|}{(d-2)n+2}
  = O\left(\frac{\dmax}{(d-2)n}\right).
  \end{align*}
Therefore
\[ \left|\lambda(\xvec) + \lambda(\xvec)^2 - \( \lambda(\xvec') - \lambda(\xvec')^2\)\right|
   = O\left(\frac{\dmax^2}{(d-2)n}\right).
\]
Now we consider $\barg$. 
Suppose that $\yvec$ is a vector which disagrees with $\xvec$ in precisely one
position, say $y_i=x_i+ \zeta$ where $\zeta\in\{-1,1\}$.  Then
using (\ref{g-useful}) (most conveniently in the form in \eqref{fred0}),
\begin{align*}
| \barg(\yvec) -\barg(\xvec)| &\leq \frac{1}{(n-2)\, ((d-2)n+2)}\, \sum_{j:j\ne i}
  (d_j-x_j)  \, \left| (d_i - x_i) - (x_i + x_j - 2) + \zeta\right| \\
&= O\left(\frac{\dmax}{n}\right) = O\left(\frac{\dmax^2}{(d-2)n}\right).  
\end{align*}
Applying this twice gives a bound of the same magnitude on $|\barg(\xvec') - \barg(\xvec)|$,
completing the proof.
\end{proof}

Now we apply Corollary~\ref{subsets} to prove the following. 

\begin{lemma}
Under the conditions of Theorem~\ref{main-sparse},
\label{eEfg}
\[
\E\left(e^{g(\X)}\right)=\exp\left(\E g(\X)
   +O\!\left(\frac{\dmax^4}{(d-2) n}\right)\right).
\]
\end{lemma}

\begin{proof}
We will apply Corollary~\ref{subsets} to $h(B)=g(\X(B))$, where the
random set $B$ is defined in Lemma~\ref{useful-later}.
We set $N=(d-1)n$ and $r=n-2$. Lemma~\ref{f+g} says that $h$ changes by at 
most $\alpha=O(\dmax^2/((d-2)n))$ if two entries of
the vector change by 1 (one increasing and one decreasing). 
The value of the error term given by Corollary~\ref{subsets} also depends on
$\min\{r,N-r\} = \min\{n-2,(d-2)n+2\}$.  We consider two cases.

If $(n-2) \leq (d-2)n+2$ then Corollary~\ref{subsets} gives
\begin{align*}
\E(e^{g(\X)}) &=\exp\left(\E g(\X) +O\!\left(\frac{\dmax^4 (n-2)}{(d-2)^2 n^2}\right)\right)
 \\*
 &= \exp\left(\E g(\X) +O\!\left(\frac{\dmax^4}{(d-2) n}\right)\right).
\end{align*}
(The second equality follows since in this case $d-2 \geq 1 - \dfrac{4}{n}\geq \dfrac{1}{2}$.)

Otherwise it holds that $(d-2)n+2 < n-2$, and here Corollary~\ref{subsets} says that
\begin{align*}
\E(e^{g(\X)}) &=\exp\left(\E g(\X) +O\!\left(\frac{\dmax^4((d-2)n+2)}{(d-2)^2 n^2}\right)\right)
  \\ 
 &= \exp\left(\E g(\X) +O\!\left(\frac{\dmax^4}{(d-2) n}\right)\right),
\end{align*}
as required.
\end{proof}

To approximate $\E g(\X)$, we need to be able to compute joint moments of the form 
$\E\big((X_j-1)_s\, (X_k-1)_t \big) $, 
where $\X=\X(B) = (X_1,\ldots, X_n)$. 
The random vector 
\[ \X - (1,1,\ldots, 1) = (X_1-1,\ldots, X_n-1)
\]
 has a multivariate hypergeometric distribution,
and from this it follows that the entries $X_1,\ldots, X_n$ of $\X=\X(B)$ satisfy
\begin{equation}
\label{joint}
 \E((X_i-1)_s(X_j-1)_t)=(d_i-1)_s (d_j-1)_t \, \frac{(n-2)_{s+t} }{ ((d-1)n)_{s+t}}
\end{equation}
for $i\neq j$. See for example~\cite[Equation (39.6)]{JKB}.

We now find an asymptotic expression for $\E(g(\X))$.

\begin{lemma}
\label{g-expression}
Under the conditions of Theorem~\ref{main-sparse},
\[ \E(g(\X)) = 
   \frac{6d^2-14d+7}{4(d-1)^2}
       + \frac{R}{2(d-1)^3}
       + \frac{(2d^2-4d+1)R^2}{4(d-1)^4\, d^2} 
         + O\biggl(\frac{d_\mx^3}{dn}\biggr).
\]
\end{lemma}

\begin{proof}
First we estimate $\E\barg(\X)$ using (\ref{g-useful}). By (\ref{joint}) we have
\begin{align*}
&\frac{\E( (X_j-1)(d_j-X_j)(d_k-X_k) )}{(n-2)((d-2)n+2)} \\[-1ex]
 &\qquad = \frac{1}{(n-2)((d-2)n+2)}\, \Big( (d_j-2)(d_k-1)\E(X_j-1)
   \\ & \hspace*{12em} {} - (d_j-2)\E((X_j-1)(X_k-1))  \\ 
   & \hspace*{12em} {}  - (d_k-1)\E((X_j-1)_2)
      + \E( (X_j-1)_2 (X_k-1))\Big) \\
       &\qquad  = \frac{(d_j-1)_2 (d_k-1)}{(n-2)((d-2)n+2)}\, \bigg( \frac{n-2}{(d-1)n}
   - \frac{2(n-2)_2}{((d-1)n)_2}  
      + \frac{(n-2)_3}{((d-1)n)_3} \bigg) \\
  &\qquad  = \frac{(d_j-1)_2\, (d_k-1)\, ((d-2)n + 1)}{((d-1)n)_3}\\
  &\qquad  = (d_j-1)_2\, (d_k-1)\, \left(\frac{d-2}{(d-1)^3 n^2} + O\left(\frac{1}{d^3n^3}\right)\right).
\end{align*}
Now
\begin{align*} \sum_{j\neq k} (d_j-1)_2 (d_k-1)
  &= \sum_{j=1}^n (d_j-1)_2 \bigl( (d-1)n - (d_j-1)\bigr)\\*
   &= (d-1) (R + (d-1)_2) n^2 + O(\dmax^2 dn).
   \end{align*}
Hence the expected value of $\barg(\X)$ is given by
\begin{align}
  \E \barg(\X) &= \left(\frac{d-2}{(d-1)^3 n^2} + O\left(\frac{1}{d^3n^3}\right)\right)\,
      \sum_{j\neq k} (d_j-1)_2 \, (d_k-1) \nonumber \\
    &= \left(\frac{d-2}{(d-1)^3 n^2} + O\left(\frac{1}{d^3n^3}\right)\right)\,
        \left( (d-1)(R + (d-1)_2)\, n^2 + O(\dmax^2 dn)\right) \nonumber \\
  &=  \frac{(d-2)(R + (d-1)_2)}{(d-1)^2} + O\left(\frac{\dmax^2}{dn}\right).\label{emu}
\end{align}

Next, recall that
\[ 
  \lambda(\X ) =\frac1{2((d-2)n+2)}\, \sum_{j=1}^n (d_j-X_j)_2.
\]
Applying (\ref{joint}) shows that
\begin{align*}
&  \frac{\E( (d_j - X_j)_2 )}{2((d-2)n+2)}\\
 &\qquad = \frac{1}{2((d-2)n+2)}\,\left( (d_j-1)_2 - 2(d_j-2)\E(X_j-1 ) + \E((X_j-1)_2 )\right)\\
   &\qquad = \frac{(d_j-1)_2}{2((d-2)n+2)}\,\left( 1 - \frac{2(n-2)}{(d-1)n} + \frac{(n-2)_2}{((d-1)n)_2}\right)\\
     &\qquad = \frac{(d_j-1)_2\, ((d-2)n+1)}{2 ((d-1)n)_2}   \\
   & \qquad = (d_j-1)_2\, 
  \left(\frac{(d-2)}{2 (d-1)^2 n} + O\left(\frac{1}{d^2n^2}\right)\right). 
\end{align*}
Therefore
\begin{align}
\E( \lambda(\X) ) &= 
 \sum_{j=1}^n (d_j-1)_2\,
  \left(\frac{(d-2)}{2 (d-1)^2 n} + O\left(\frac{1}{d^2n^2}\right)\right) \nonumber\\
  &= \frac{(d-2)(R + (d-1)_2)}{2 (d-1)^2 } + O\left(\frac{\dmax}{dn}\right).
\label{f1}
\end{align}
The same approach works for $\E(\lambda(\X)^2)$ but the details
are a little messier.  Observe that
\begin{align}
&\lambda(\X)^2 \nonumber \\*[-1ex]
 &= \frac{1}{4((d-2)n + 2)^2}\,
  \left(\left(\sum_{j\neq k} (d_j-X_j)_2 (d_k-X_k)_2 \right)
    + \sum_{j=1}^n (d_j-X_j)^2 (d_j-X_j-1)^2\right).
\label{expand}
\end{align}
Applying (\ref{joint}) to the off-diagonal summands gives
\begin{align*}
&\frac{\E( (d_j-X_j)_2 (d_k-X_k)_2 )}{4((d-2)n+2)^2}\,  \\
 &\quad= \frac{1}{4((d-2)n+2)^2}\, \Big(
      (d_j-1)_2(d_k-1)_2 - 2(d_j-1)_2 (d_k-2) \E(X_k-1)\\ & \qquad \quad {} - 2(d_j-2)(d_k-1)_2 \E(X_j-1)   + (d_j-1)_2 \E((X_k-1)_2 ) \\ & \qquad \quad {}+ 4(d_j-2)(d_k-2) \E((X_j-1)(X_k-1)) 
   + (d_k-1)_2 \E((X_j-1)_2 ) \\ & \qquad \quad {} 
  - 2(d_j-2)\, \E( (X_j-1)(X_k-1)_2 )  - 2(d_k-2)\, \E((X_j-1)_2(X_k-1))  \\ 
  & \hspace*{95mm} {} + \E((X_j-1)_2(X_k-1)_2) \Big)\\
  &\quad = \frac{(d_j-1)_2 (d_k-1)_2}{4((d-2)n+2)^2}\,
    \left( 1 - \frac{4(n-2)}{(d-1)n} + \frac{6(n-2)_2}{((d-1)n)_2} - \frac{4(n-2)_3}{((d-1)n)_3}
       + \frac{(n-2)_4}{((d-1)n)_4}\right)\\
   &\quad = 
   \frac{(d_j-1)_2 (d_k-1)_2 \,((d-2)n+1)_3}{4((d-1)n)_4 ((d-2)n+2)}\\
   &\quad = (d_j-1)_2 (d_k-1)_2 \, \left( \frac{(d-2)^2\, }{4(d-1)^4 n^2} + O\left(\frac{1}{d^3n^3}\right)\right).
   \end{align*}
   Next, calculate
\[ \sum_{j\neq k} (d_j-1)_2 (d_k-1)_2 = (R + (d-1)_2)^2  n^2 + 
    O(\dmax^3 dn).\]
Therefore the contribution to $\lambda(\X )^2$ from the off-diagonal summands is
\begin{align*}
& \left(\frac{(d-2)^2}{4(d-1)^4 n^2} + O\left(\frac{1}{d^3n^3}\right)\right)\, 
\sum_{j \neq k} (d_j-1)_2 (d_k-1)_2\\
   &=  \left(\frac{(d-2)^2}{4(d-1)^4 n^2 } + O\left(\frac{1}{d^3 n^3}\right)\right)\, 
      \Bigl( (R + (d-1)_2)^2 n^2 + O\left(\dmax^3 d n\right)\Bigr)\\
    &= \frac{(d-2)^2 (R+(d-1)_2)^2}{4(d-1)^4} + O\left(\frac{\dmax^3}{dn}\right).
\end{align*}
The contribution to $\E(\lambda(\X)^2 )$ from the diagonal
terms of (\ref{expand}) (that is, the second summation in (\ref{expand})) is
\begin{align*}
    \frac{1}{4((d-2)n + 2)^2} \, \sum_{j=1}^n \E( (d_j-X_j)^2 (d_j-X_j-1)^2 )
  &= O\left(\frac{\dmax^2}{(d-2) n}\right)\, \E(\lambda(\X))\\
  &= O\left(\frac{\dmax^3}{d n}\right), 
\end{align*}
using (\ref{f1}).
Therefore
\begin{equation}
\E(\lambda(\X)^2 ) = \frac{(d-2)^2\, (R+ (d-1)_2)^2}{4(d-1)^4}
   + O\left(\frac{\dmax^3}{dn}\right).
\label{f2}
\end{equation}
The result follows by combining (\ref{lambda0}), (\ref{emu}), (\ref{f1}) and (\ref{f2}),
after some rearranging.
\end{proof}

Now we may easily prove our main theorem.

\begin{proof}[Proof of Theorem~\ref{main-sparse}.]
The number of graphs with degree sequence $\dvec$ is positive when $n$ is sufficiently large,
by (\ref{Nd}).  That is, $\dvec$ is graphical for sufficiently large $n$.
The claimed asymptotic expression for $\E \tau_{\dvec}$ then follows immediately from 
Lemmas~\ref{useful-later},~\ref{eEfg} and~\ref{g-expression}. We also briefly justify the bound
\begin{align*}
  \eta &= \min\biggl\{ \frac{d_\mx^4}{(d-2)^2n},\,
                                 \frac{d_\mx^3\log n}{(d-2)n},\,
                                 d_\mx(d-2)\biggr\} \\
         &= O\biggl( \frac{d_\mx^4}{(d-2)n} + \frac{(\log n)^{5/2}}{n^{1/2}}\biggr).
\end{align*}
Note that $(d_\mx^3\log n)/((d-2)n)\le d_\mx^4/((d-2)n)$ if $\dmax\ge \log n$. 
When $\dmax\leq \log n$, take the geometric mean of $(d_\mx^3\log n)/((d-2)n)$ and $d_\mx(d-2)$.
\end{proof}

\subsection*{Acknowledgements}
We would like to thank the referees for their helpful comments.

\end{document}